\documentclass[preprint,12pt]{elsarticle}
\usepackage[top=3cm, bottom=3cm, left=2.2cm, right=2.2cm]{geometry}
\usepackage{amssymb}    
\usepackage{amsmath}    
\usepackage{amsthm}     
\usepackage{mathtools}  
\usepackage{cases}      
\usepackage{hyperref}   
\usepackage{multirow}  
\usepackage{array}     
\usepackage{makecell}  
\usepackage{ragged2e}  
\usepackage{adjustbox}   

\usepackage[table]{xcolor}

\newtheorem{theorem}{Theorem}[section]    
\newtheorem{lemma}[theorem]{Lemma}        
\newtheorem{definition}[theorem]{Definition}  
\newtheorem{proposition}[theorem]{Proposition} 

\newcommand{\R}{{\mathbb R}}

\begin{document}

\begin{frontmatter}

\title{On a two-species Keller-Segel model with degenerate diffusion and two stimuli}

\author{Shen Bian \corref{cor1}}
\cortext[cor1]{Corresponding author. E-mail: bianshen66@163.com} 

\affiliation{organization={Department of Mathematics, Beijing University of Chemical Technology},
            addressline={No. 15 East North Third Ring Road, Chaoyang District},
            city={Beijing},
            postcode={100029},
            country={China}}


\begin{abstract}
This paper investigates a two-species chemotaxis system with degenerate diffusion in the whole space $\R^d(d \ge 3)$. The diffusion of each species is governed by porous-medium-type operators. Under suitable conditions on the diffusion and aggregation exponents, we establish the global existence of uniformly bounded weak solutions. The proof hinges on a refined energy estimate that exploits the regularizing effect of degenerate diffusion to counteract the chemotactic aggregation. Furthermore, under additional assumptions on the system parameters, we derive exponential convergence rates of the solutions toward the constant steady state. Our results reveal that sufficiently strong degenerate diffusion ensures global boundedness and exponential stabilization in the whole space.
\end{abstract} 

\end{frontmatter} 

\section{Introduction}

Chemotaxis refers to the mechanism by which cells move in response to external signals, playing a fundamental role in various biological processes such as morphogenesis, wound healing, tumor growth, and cell differentiation. To describe the chemotactic phenomenon, the most classical model for the collective motion of cells was developed by Patlak \cite{P53} and Keller-Segel \cite{KS70} and is written as
\begin{align}
\left\{
  \begin{array}{ll}
  u_t=\Delta u-\chi \nabla \cdot (u \nabla c)+f(u), & x \in \Omega, t>0, \\
\varepsilon c_t=\Delta c-c+u, & x \in \Omega, t>0,
  \end{array}
\right.
\end{align}
where $\Omega$ is assumed to be a bounded domain or the whole space. This system consists of two equations that describe the evolution of the density of bacteria $u(x,t)$ and the chemoattractant concentration $c(x,t)$. The chemical attractant is emitted by the cells themselves, and it also diffuses in space and decays in time. The logistic-type source $f(u)$ describes the proliferation and density-dependent saturation of bacteria. This system has been widely studied, and a comprehensive survey of the mathematical results was given by Horstmann \cite{HP09,Horst1,Horst2}. One of the most famous research results concerns the so-called chemotactic collapse \cite{BC99,CP06,bp07,bdp06,dp04,JL92,HV96}.

For two-species communities, such dynamics are particularly pronounced: each species may secrete specific signaling molecules, while simultaneously exhibiting directed movement toward the signals produced by themselves or by the other species. This coupling of interspecific interactions and chemotactic response forms the core of two-species chemotaxis models. In the past two decades, two-species chemotaxis models have attracted much interest. When two species secrete the same chemical signal, the system reads
\begin{align}\label{07061}
\left\{
  \begin{array}{ll}
    u_t=d_1 \Delta u-\chi_1 \nabla \cdot (u \nabla c)+\mu_1 u(1-u-a_1 v), & x \in \Omega, t>0, \\
    v_t=d_2 \Delta v-\chi_2 \nabla \cdot (v \nabla c)+\mu_2 v(1-v-a_2 u), & x \in \Omega, t>0, \\
    \varepsilon c_t=\Delta c-c+\alpha u+\beta v, & x \in \Omega, t>0, 
  \end{array}
\right.
\end{align} 
where $u(x,t)$ and $v(x,t)$ represent the densities of two distinct types of cells, and $c$ is generated by both $u$ and $v$. The global well-posedness and asymptotically stable constant steady state of \eqref{07061} were established by \cite{M18,ST14,TW12} for the parabolic-elliptic system ($\varepsilon=0$) and by \cite{B16,WMH18,WZ15} for the parabolic-parabolic system ($\varepsilon>0$). When the two species respond to chemicals emitted by each other rather than to their own signals, the system takes the form
\begin{align}\label{07062}
\left\{
  \begin{array}{ll}
    u_t=d_1 \Delta u-\chi_1 \nabla \cdot (u \nabla c)+\mu_1 u(1-u-a_1 v), & x \in \Omega, t>0, \\
   \tau c_t=\Delta c-c+v, & x \in \Omega, t>0, \\
    v_t=d_2 \Delta v-\chi_2 \nabla \cdot (v \nabla z)+\mu_2 v(1-v-a_2 u), & x \in \Omega, t>0, \\
     \tau z_t=\Delta z-z+u, & x \in \Omega, t>0, 
  \end{array}
\right.
\end{align} 
where $c(x,t)$ and $z(x,t)$ describe the concentrations of chemical attractants emitted by $v(x,t)$ and $u(x,t)$, respectively. For $\tau=0,$ when $\chi_1,\chi_2$ are below certain thresholds, the system is globally well-posed and converges towards the constant steady state of \eqref{07062} \cite{CN18,ZL17,WM20}. Moreover, \cite{WZ20,TW15} studied the system when $c$ and $z$ depend on $u$ and $v$ in a nonlinear way. If $\tau>0,$ the global well-posedness and finite-time blow-up of the corresponding parabolic-parabolic system were established by \cite{LX20,PW20,MT25,LTW20}. 

In this work, we consider the two-species chemotaxis model with degenerate diffusion on the whole space $\R^d$
\begin{align}\label{uvsystem}
\left\{
  \begin{array}{ll}
    u_t=\Delta u^{m_1}-\nabla \cdot (u^{\sigma_1} \nabla c)+ u(N_1-u), & x \in \R^d, t > 0, \\
    -\Delta c+c=a_{11}u+a_{12} v,  & x \in \R^d, t > 0, \\
    v_t=\Delta v^{m_2}-\nabla \cdot (v^{\sigma_2} \nabla z)+ v(N_2-v), & x \in \R^d, t > 0, \\ 
    -\Delta z+z=a_{21}u+a_{22}v, & x \in \R^d, t > 0, \\
    (u,v)(x,0)=(u_0,v_0)(x) \ge 0, & x \in \R^d,   
  \end{array}
\right.
\end{align}
where $m_1>1,m_2>1,a_{ij}>0, i,j=1,2$. Compared to the comprehensive study of the linear diffusion case ($m_1=m_2=1$), the degenerate diffusion system is not well-understood, and only a few results are devoted to global bounded solutions in bounded domains ($\sigma_1=\sigma_2=1$), see \cite{LL22,Z21,Z17}. The goal of this paper is to investigate the influence of the nonlinear diffusion and aggregation exponents $m_i,\sigma_i (i=1,2)$ on the solutions, in particular, whether they affect the global well-posedness and long-time behavior of solutions. Initial data will be assumed throughout this paper to satisfy
\begin{align}\label{initialdata}
(u_0,v_0) \in L^1 \cap L^\infty(\R^d).
\end{align}

Before proceeding, we first state the definition of the weak solution $(u,v)$ considered in this paper.
\begin{definition}
Let $(u_0,v_0)$ be the initial data satisfying \eqref{initialdata} and $T \in (0,\infty].$ A pair of weak solutions $(u,v)$ to \eqref{uvsystem} is a pair of non-negative functions $(u,v) \in L^\infty \left(0,T;L_+^1 \cap L^\infty(\R^d) \right)$ such that for any $0<t<T$ and all test functions $\psi_1, \psi_2 \in C_0^\infty(\R^d)$, 
 \begin{align*}
 &\int_{\R^d} \psi_1 u(\cdot,t)dx-\int_{\R^d} \psi_1 u_0(x) dx =\int_0^t
 \int_{\R^d} u^{m_1} \Delta \psi_1 dx ds \nonumber \\
 & + \int_0^t \int_{\R^d} u^{\sigma_1} \nabla c \cdot \nabla \psi_1 dx ds+ \int_0^t \int_{\R^d} \psi_1 u(N_1-u) dxds, \nonumber \\
 &\int_{\R^d} \psi_2 v(\cdot,t)dx-\int_{\R^d} \psi_2 v_0(x) dx =\int_0^t
 \int_{\R^d} v^{m_2} \Delta \psi_2 dx ds \nonumber \\
 & + \int_0^t \int_{\R^d} v^{\sigma_2} \nabla z \cdot \nabla \psi_2 dx ds+ \int_0^t \int_{\R^d} \psi_2 v(N_2-v) dx ds.
 \end{align*}
\end{definition}

Our first result concerns the global existence of solutions to system \eqref{uvsystem} in dimension $d \ge 3$.
\begin{theorem}\label{th1}
Let $d \ge 3, m_1>1,m_2>1, \sigma_1>1, \sigma_2>1$. Assume that
\begin{align}\label{m1m2}
\min\{m_1,m_2\}>1+\max\{\sigma_1,\sigma_2 \}-2/d.
\end{align}
Then for any initial data satisfying \eqref{initialdata}, system \eqref{uvsystem} possesses a global weak solution $(u,v)$ that is uniformly bounded for all times.   
\end{theorem}

The proof of Theorem \ref{th1}, given in Section \ref{global}, is based on an energy estimate in which the key observation is that the mutually repulsive effect of the death and diffusion terms dominates the attractive effect of the aggregation term.

Our next result concerns the long-time behavior of solutions.
\begin{theorem}\label{th2}
Let $d \ge 3, m_1>1,m_2>1, \sigma_1=\frac{m_1+1}{2}, \sigma_2=\frac{m_2+1}{2}$. Under assumptions \eqref{initialdata} and 
\begin{align}
& m_1>\frac{m_2-1}{2}+2-\frac{2}{d} \\
\text{or} \quad & m_2>\frac{m_1-1}{2}+2-\frac{2}{d},
\end{align} 
suppose further that the positive numbers $a_{ij}$ satisfy 
\begin{align}
\left\{
  \begin{array}{ll}
    (a_{11}^2+a_{12}^2)N_1<16 m_1, \\[1mm]
    \frac{(a_{11}^2+a_{12}^2)N_1}{m_1}+\frac{(a_{21}^2+a_{22}^2)N_2}{m_2}<16+\frac{(a_{11}a_{22}-a_{12}a_{21})^2 N_1N_2}{16m_1m_2}.
  \end{array}
\right.
\end{align}
Then the global bounded solution $(u,v)$ of \eqref{uvsystem} satisfies that for some $\lambda>0$ and $C>0$,
\begin{align}
\|u(\cdot,t)-N_1\|_{L^2(\R^d)}+\|v(\cdot,t)-N_2\|_{L^2(\R^d)} \le C e^{-\lambda t},\quad \text{for all } t>0.
\end{align}
\end{theorem}

The remainder of this paper is organized as follows. Section \ref{Pre} establishes an existence criterion that characterizes the maximal existence time of solutions to \eqref{uvsystem}. Section \ref{global} then proves the global existence of solutions, thereby completing the proof of Theorem \ref{th1}. Finally, Section \ref{longtime} is devoted to the long-time behavior of the global solution and provides the proof of Theorem \ref{th2}.

In what follows, we denote by $C$ a generic constant (which may vary between lines) and by $C=C(\cdot,\cdots,\cdot)$ a constant depending only on the quantities appearing in parentheses. We also use the simplified notations $\|\cdot\|_{r}:=\|\cdot\|_{L^r(\R^d)}, 1 \le r <\infty.$

\section{Preliminaries}\label{Pre}

In this section, we first prepare the following lemma, which will play an important role in the proof of the global existence of solutions to system \eqref{uvsystem}.
\begin{lemma}[\cite{B16}]\label{GNS1}
Let $p=\frac{2d}{d-2}, 1 \le r<q<p$ and $\frac{q}{r}<\frac{2}{r}+1-\frac{2}{p}.$ Then for any $w \in L^r(\R^d) \cap H^1(\R^d)$, it holds that
\begin{align}
\|w\|_{L^q(\R^d)}^q \le C(d) C_0^{-\frac{\lambda q}{2-\lambda q}} \|w\|_{L^r(\R^d)}^\gamma +C_0 \|\nabla w\|_{L^2(\R^d)}^2.
\end{align}
Here $C(d)$ is a constant depending on $d,$ $C_0$ is an arbitrary positive constant, and 
\begin{align}
\lambda=\frac{\frac{1}{r}-\frac{1}{q}}{\frac{1}{r}-\frac{1}{p}} \in (0,1),\quad \gamma=\frac{2(1-\lambda)q}{2-\lambda q}=\frac{2 \left( 1-\frac{q}{p} \right)}{\frac{2-q}{r}+1-\frac{2}{p}}.
\end{align}
\end{lemma}

We now consider the regularized problem
\begin{align}\label{kseps}
\left\{
  \begin{array}{ll}
\partial_t u_\varepsilon=\Delta u^{m_1}_\varepsilon+\varepsilon \Delta u_\varepsilon-\nabla \cdot \left( u_\varepsilon^{\sigma_1} \nabla c_\varepsilon \right)+ u_\varepsilon(N_1-u_\varepsilon), & x \in \R^d, t>0, \\
    -\Delta c_\varepsilon+c_\varepsilon=a_{11}u_\varepsilon+a_{12} v_\varepsilon,  & x \in \R^d, t >0, \\
    \partial_t v_\varepsilon=\Delta v^{m_2}_\varepsilon+\varepsilon \Delta v_\varepsilon-\nabla \cdot \left( v_\varepsilon^{\sigma_2} \nabla z_\varepsilon \right)+ v_\varepsilon(N_2-v_\varepsilon), & x \in \R^d, t>0, \\
    -\Delta z_\varepsilon+z_\varepsilon=a_{21}u_\varepsilon+a_{22}v_\varepsilon, & x \in \R^d, t> 0, \\
    (u_\varepsilon,v_\varepsilon)(x,0)=(u_{0\varepsilon},v_{0\varepsilon})(x) \ge 0, & x \in \R^d.
  \end{array}
\right.
\end{align}
Here, $(u_{0\varepsilon},u_{0\varepsilon})$ is a sequence of approximations for $(u_0,v_0)$ that can be constructed to satisfy the following: there exists $\varepsilon_0>0$ such that for any $0<\varepsilon<\varepsilon_0,$
\begin{align}
\left\{
  \begin{array}{ll}
(u_{0\varepsilon},v_{0\varepsilon}) \ge 0, ~~ \|u_\varepsilon(x,0)\|_{1}=\|u_0\|_{1},~~ \|v_\varepsilon(x,0)\|_{1}=\|v_0\|_{1}, \\
 (u_{0\varepsilon},v_{0\varepsilon})(x) \to (u_0,v_0)(x) \text{ in } L^q(\R^d),\text{ for }1 \le q <\infty,\text{ as } \varepsilon \to 0.
  \end{array}
\right.
\end{align}
The regularized problem admits global in time smooth solutions for any $\varepsilon>0$. This approximation has been proved to be convergent. More precisely, following the arguments in \cite[Theorem 4.2]{BL14} and \cite[Section 4]{suku06}, we assert that if
\begin{align}\label{Linfinity}
\|u_\varepsilon(\cdot,t)\|_{L^\infty(\R^d)}+\|v_\varepsilon(\cdot,t)\|_{L^\infty(\R^d)}<C_0,
\end{align}
where $C_0$ is independent of $\varepsilon,$ then there exists a subsequence $\varepsilon_n \to 0$ such that
\begin{align}
  \begin{array}{ll}
    (u_{\varepsilon_n},v_{\varepsilon_n}) \rightarrow (u,v) \text{ in } L^r(0,T;L^r(\R^d)),\quad 1 \le r<\infty \label{conver1}
  \end{array}
\end{align}
and $(u,v)$ is a weak solution to \eqref{uvsystem} on $[0,T)$.

According to the above analysis, a weak solution to \eqref{uvsystem} on $[0,T)$ exists when \eqref{Linfinity} is fulfilled. Hence, we shall focus on establishing the availability of the $L^\infty$-bound. As we will observe in the following lemma, where the local-in-time existence and blow-up criteria are constructed, such a bound follows from the $L^r$ norm for $r>1$, which additionally provides a characterisation of the maximal existence time.

\begin{lemma}[Local existence and blow-up criteria]\label{ueps}
Under assumption \eqref{initialdata} on the initial condition, there exist a maximal existence time $T_w \in (0,\infty]$ and a weak solution $(u,v)$ to \eqref{uvsystem} on $[0,T_w)$. If $T_w<\infty$, then
\begin{align}\label{criterion}
\|u(\cdot,t)\|_{L^\infty(\R^d)}+ \|v(\cdot,t)\|_{L^\infty(\R^d)} \to \infty \text{ as } t \to T_w.
\end{align}
\end{lemma}
\begin{proof}
Integrating the first and third equations of \eqref{kseps} over $\R^d$ yields 
\begin{align}
\frac{d}{dt} \int_{\R^d} u_\varepsilon dx \le N_1 \int_{\R^d} u_\varepsilon dx, \quad \frac{d}{dt} \int_{\R^d} v_\varepsilon dx \le N_2 \int_{\R^d} v_\varepsilon dx.
\end{align}
Applying Gronwall's inequality, we obtain
\begin{align}
\int_{\R^d} u_\varepsilon dx \le \int_{\R^d} u_{0\varepsilon} dx ~e^{N_1 t}, \quad 0<t<\infty, \\
\int_{\R^d} v_\varepsilon dx \le \int_{\R^d} v_{0\varepsilon} dx ~e^{N_2 t}, \quad 0<t<\infty. 
\end{align}
We now turn to the $L^k$-estimates for $k>1$, which will be essential for the $L^\infty$-bound. 

{\it\textbf{Step 1}} ($L^k$-estimates, $k \in (1,\infty)$) \quad Multiplying the first equation of \eqref{kseps} by $ku_\varepsilon^{k-1}$ and the third equation of \eqref{kseps} by $kv_\varepsilon^{k-1}$, we obtain
\begin{align}\label{230127}
&\frac{d}{dt} \int_{\R^d} \left(u_\varepsilon^k+v_\varepsilon^k \right) dx +\frac{4k(k-1)m_1}{(k+m_1-1)^2} \int_{\R^d} \left| \nabla u_\varepsilon^{\frac{k+m_1-1}{2}} \right|^2 dx+ \frac{4k(k-1)m_2}{(k+m_2-1)^2} \int_{\R^d} \left| \nabla v_\varepsilon^{\frac{k+m_2-1}{2}} \right|^2 dx \nonumber \\
&+k\int_{\R^d} u_\varepsilon^{k+1} dx+\frac{k(k-1)}{k+\sigma_1-1}\int_{\R^d} u_\varepsilon^{k+\sigma_1-1} c dx+\varepsilon \frac{4(k-1)}{k}\int_{\R^d} \left| \nabla u_\varepsilon^{\frac{k}{2}} \right|^2 dx  \nonumber \\
&+k\int_{\R^d} v^{k+1} dx+
\frac{k(k-1)}{k+\sigma_2-1}\int_{\R^d} v_\varepsilon^{k+\sigma_2-1} z dx+\varepsilon \frac{4(k-1)}{k}\int_{\R^d} \left| \nabla v_\varepsilon^{\frac{k}{2}} \right|^2 dx \nonumber \\
= &N_1 k\int_{\R^d} u^k dx+\frac{k(k-1)}{k+\sigma_1-1} a_{11} \int_{\R^d} u_\varepsilon^{k+\sigma_1} dx+ \frac{k(k-1)}{k+\sigma_1-1} a_{12} \int_{\R^d} u_\varepsilon^{k+\sigma_1-1} v_\varepsilon dx \nonumber \\
&+N_2 k\int_{\R^d} v_\varepsilon^k dx+\frac{k(k-1)}{k+\sigma_2-1} a_{22} \int_{\R^d} v_\varepsilon^{k+\sigma_2} dx+ \frac{k(k-1)}{k+\sigma_2-1} a_{21} \int_{\R^d} v_\varepsilon^{k+\sigma_2-1} u_\varepsilon dx.
\end{align}
An application of Young's inequality yields
\begin{align}
\int_{\R^d} u_\varepsilon^{k+\sigma_1-1} v_\varepsilon dx \le C \int_{\R^d} u_\varepsilon^{k+\sigma_1}dx+C \int_{\R^d} v_\varepsilon^{k+\sigma_1}dx
\end{align}
and
\begin{align}
\int_{\R^d} v_\varepsilon^{k+\sigma_2-1} u_\varepsilon dx \le C \int_{\R^d} v_\varepsilon^{k+\sigma_2}dx+C \int_{\R^d} u_\varepsilon^{k+\sigma_2}dx.
\end{align}
Substituting the above two inequalities into \eqref{230127}, we obtain
\begin{align}\label{07070}
&\frac{d}{dt} \int_{\R^d} (u_\varepsilon^k+v_\varepsilon^k) dx +\frac{4k(k-1)m_1}{(k+m_1-1)^2} \int_{\R^d} \left| \nabla u_\varepsilon^{\frac{k+m_1-1}{2}} \right|^2 dx+ \frac{4k(k-1)m_2}{(k+m_2-1)^2} \int_{\R^d} \left| \nabla v_\varepsilon^{\frac{k+m_2-1}{2}} \right|^2 dx \nonumber \\
& +k\int_{\R^d} u_\varepsilon^{k+1} dx+k\int_{\R^d} v^{k+1} dx \nonumber\\
\le & N_1 k\int_{\R^d} u^k dx+N_2 k\int_{\R^d} v_\varepsilon^k dx+C \int_{\R^d} u_\varepsilon^{k+\sigma_1}dx+C \int_{\R^d} v_\varepsilon^{k+\sigma_1}dx \nonumber \\
& +C \int_{\R^d} v_\varepsilon^{k+\sigma_2}dx+C \int_{\R^d} u_\varepsilon^{k+\sigma_2}dx.
\end{align}
Letting
\begin{align}\label{07071}
w=u_\varepsilon^{\frac{k+m_1-1}{2}},\quad q=\frac{2(k+\sigma_1)}{k+m_1-1},\quad r=\frac{2k}{k+m_1-1}
\end{align}
in Lemma \ref{GNS1} with $k>\frac{2 \sigma_1-p(m_1-1)}{p-2}$ (such that $q<p$) and $\frac{p}{p-2}(\sigma_1-m_1+1)<k<k+\sigma_1$ (such that $\frac{q}{r}<\frac{2}{r}+1-\frac{2}{p}$), we have
\begin{align}\label{07072}
C \int_{\R^d} u_\varepsilon^{k+\sigma_1} dx \le \frac{k(k-1)m_1}{(k+m_1-1)^2} \left\| \nabla u_\varepsilon^{\frac{k+m_1-1}{2}} \right\|_2^2+C(k) \|u_\varepsilon\|_{k}^{b_1},
\end{align}
where
\begin{align}
b_1=\frac{(k+\sigma_1)(1-\lambda)}{1-\frac{\lambda q}{2}},\quad \lambda=\frac{\frac{k+m_1-1}{2k}-\frac{k+m_1-1}{2(k+\sigma_1)}}{\frac{k+m_1-1}{2k}-\frac{1}{p}}.
\end{align}
A direct computation yields
\begin{align}\label{07073}
b_1=k \frac{\frac{k+m_1-1}{2}-\frac{k+\sigma_1}{p}}{\frac{k+m_1-1}{2}-\frac{k}{p}-\frac{\sigma_1}{2}}>k.
\end{align}
Following arguments as in \eqref{07071}-\eqref{07073}, we also have
\begin{align}\label{07081}
C \int_{\R^d} u_\varepsilon^{k+\sigma_2} dx \le \frac{k(k-1)m_1}{(k+m_1-1)^2} \left\| \nabla u_\varepsilon^{\frac{k+m_1-1}{2}} \right\|_2^2+C(k) \|u_\varepsilon\|_{k}^{b_2},
\end{align} 
where $k>\frac{p}{p-2}(\sigma_2-m_1+1)$ and 
\begin{align}\label{07130}
b_2=k \frac{\frac{k+m_1-1}{2}-\frac{k+\sigma_2}{p}}{\frac{k+m_1-1}{2}-\frac{k}{p}-\frac{\sigma_2}{2}}>k.
\end{align}
Similarly,  applying the arguments from \eqref{07071}-\eqref{07130} to $v_\varepsilon$, we obtain
\begin{align}\label{07082}
& C \int_{\R^d} v_\varepsilon^{k+\sigma_1} dx+C \int_{\R^d} v_\varepsilon^{k+\sigma_2} dx \nonumber \\
\le & \frac{k(k-1)m_2}{(k+m_2-1)^2} \left\| \nabla v_\varepsilon^{\frac{k+m_2-1}{2}} \right\|_2^2+C(k) \|v_\varepsilon\|_{k}^{a_1}+\frac{k(k-1)m_2}{(k+m_2-1)^2} \left\| \nabla v_\varepsilon^{\frac{k+m_2-1}{2}} \right\|_2^2+C(k) \|v_\varepsilon\|_{k}^{a_2},
\end{align}
where $k> \frac{p}{p-2} \left(\max\{\sigma_1,\sigma_2\}-m_2+1 \right)$ and
\begin{align}
a_1=k \frac{\frac{k+m_2-1}{2}-\frac{k+\sigma_1}{p}}{\frac{k+m_2-1}{2}-\frac{k}{p}-\frac{\sigma_1}{2}}>k, \\
a_2=k \frac{\frac{k+m_2-1}{2}-\frac{k+\sigma_2}{p}}{\frac{k+m_2-1}{2}-\frac{k}{p}-\frac{\sigma_2}{2}}>k.
\end{align}
Therefore, substituting \eqref{07072}, \eqref{07081} and \eqref{07082} into \eqref{07070}, we conclude that
\begin{align}\label{07085}
& \frac{d}{dt} \int_{\R^d} (u_\varepsilon^k+v_\varepsilon^k)dx+ \frac{k(k-1)m_1}{(k+m_1-1)^2} \int_{\R^d} \left| \nabla u_\varepsilon^{\frac{k+m_1-1}{2}} \right|^2 dx+ \frac{k(k-1)m_2}{(k+m_2-1)^2} \int_{\R^d} \left| \nabla v_\varepsilon^{\frac{k+m_2-1}{2}} \right|^2 dx  \nonumber \\
\le & C \int_{\R^d} (u_\varepsilon^k+v_\varepsilon^k)dx+C(k) \|u_\varepsilon\|_{k}^{b_1}+C(k) \|u_\varepsilon\|_{k}^{b_2}+ C(k) \|v_\varepsilon\|_{k}^{a_1}+C(k) \|v_\varepsilon\|_{k}^{a_2} \nonumber \\
\le & C+ \left(\|u_\varepsilon\|_k^k+\|v_\varepsilon\|_k^k \right)^a,
\end{align}
where 
\begin{align*}
a &=\max \left\{ \frac{b_1}{k},\frac{b_2}{k},\frac{a_1}{k},\frac{a_2}{k} \right\} \\
&=\max\left\{ 1+ \frac{M}{k-\frac{d(M-m_1+1)}{2}},1+ \frac{M}{k-\frac{d(M-m_2+1)}{2}} \right\}, \quad M=\max\{\sigma_1,\sigma_2\}.
\end{align*}
As a result, the $L^k$-norm for $k>\frac{d}{2}\left(\max(\sigma_1,\sigma_2)-\min(m_1,m_2)+1\right)$ is bounded locally in time:
\begin{align}\label{nablaLr}
\|u_\varepsilon\|_k^k+\|v_\varepsilon\|_k^k \le \frac{C(k,d)}{(T_k-t)^{\frac{1}{a-1}}}, \quad T_k=(\|u_\varepsilon\|_k^k+\|v_\varepsilon\|_k^k)^{1-a}.
\end{align}
Moreover, returning to \eqref{07085}, we also deduce that
\begin{align}\label{07086}
\left\|\nabla u_\varepsilon^{\frac{k+m_1-1}{2}}\right\|_{L^2(0,T_k;L^2(\R^d))}+\left\|\nabla v_\varepsilon^{\frac{k+m_2-1}{2}}\right\|_{L^2(0,T_k;L^2(\R^d))} \le C\left( \|u_{0\varepsilon}\|_{k}, \|v_{0\varepsilon}\|_{k} \right).
\end{align}

{\it\textbf{Step 2}} ($L^\infty$-estimates)\quad As a direct result of Step 1, by applying the Moser iterative method, we have
\begin{align}\label{Linfinity23}
\displaystyle \sup_{0<t<T_k} \left(\|u_\varepsilon(\cdot,t)\|_{L^\infty(\R^d)}+ \|v_\varepsilon(\cdot,t)\|_{L^\infty(\R^d)}\right) \le C\left(\|u_{0\varepsilon}\|_{L^1\cap L^\infty(\R^d)}, \|v_{0\varepsilon}\|_{L^1\cap L^\infty(\R^d)} \right)
\end{align}
by a word-for-word translation of the proof of \cite[Theorem 4.2]{BL14}. Armed with the regularities \eqref{nablaLr}, \eqref{07086} and \eqref{Linfinity23}, it directly follows that
\begin{align}
\|\nabla u_\varepsilon \|_{L^2(0,T_k;L^2(\R^d))}+\|\nabla v_\varepsilon \|_{L^2(0,T_k;L^2(\R^d))} \le C, \\
\|u_\varepsilon^{\sigma_1} \nabla c_\varepsilon \|_{L^\infty(0,T_k;L^\infty(\R^d))} + \|v_\varepsilon^{\sigma_2} \nabla z_\varepsilon \|_{L^\infty(0,T_k;L^\infty(\R^d))} \le C.
\end{align}
Here, $C$ are constants depending only on $\|u_{0\varepsilon}\|_{L^1\cap L^\infty(\R^d)}$ and $\|v_{0\varepsilon}\|_{L1 \cap L^\infty(\R^d)}$. As a consequence, the a priori bounds in the theorem hold true uniformly in $\varepsilon$. Thus we can pass to the limit $(u_\varepsilon,v_\varepsilon) \to (u,v)$ as $\varepsilon \to 0$ (without relabeling) by the Lions-Aubin lemma, which provides the required time compactness. In the limit, we obtain the local existence of a weak solution $(u,v)$ to \eqref{uvsystem}. Finally, Following the argument in the proof of \cite[Theorem 2.4]{BJ09}, we obtain the characterization of the maximal existence time, and thus complete the proof. \quad $\square$
\end{proof}

\section{Global existence: proof of Theorem \ref{th1}}\label{global}

In this section, under suitable conditions on $m_1,m_2$ and $\sigma_1,\sigma_2$, we extend the local existence theory in Lemma \ref{ueps} to the following global existence result.

\begin{proposition}\label{prop1}
Let $d \ge 3, m_1>1,m_2>1, \sigma_1>1, \sigma_2>1$, and let $(u,v)$ be the weak solution obtained in Lemma \ref{ueps} on $[0,T_w)$. Under the condition 
\begin{align}
\min\{m_1,m_2\}>1+\max\{\sigma_1,\sigma_2 \}-2/d,
\end{align}
the weak solution exists globally in time, i.e. $T_w=\infty.$
\end{proposition}
\begin{proof}
We divide the proof into three steps. First, by imposing suitable conditions on $m_1,m_2$, we show that the $L^r$-norm of the solution is uniformly bounded for all $t>0.$ Second, an iterative argument yields the uniform boundedness of the solution in time. Finally, we prove that the local weak solution obtained in Lemma \ref{ueps} exists globally.

We first establish the $L^1$-norm of solutions, which will be used in the derivation of the $L^r$-bounds for $1<r<\infty.$ Integrating the first and third equations of \eqref{uvsystem} over $\R^d$, we obtain
\begin{align*}
\frac{d}{dt}\int_{\R^d} udx \le N_1 \int_{\R^d} u dx, \\
\frac{d}{dt}\int_{\R^d} vdx \le N_2 \int_{\R^d} v dx,
\end{align*}
which implies
\begin{align}\label{um1vm1}
\int_{\R^d} udx \le \|u_0\|_1 e^{N_1 t}, \quad 0<t<\infty, \\
\int_{\R^d} vdx \le \|v_0\|_1 e^{N_2 t}, \quad 0<t<\infty. 
\end{align}

Now we turn to the $L^k$-estimates for $k>1$. Multiplying the first equation of \eqref{uvsystem} by $k u^{k-1}$ and the third equation of \eqref{uvsystem} by $kv^{k-1}$, we obtain
\begin{align*}
&\frac{d}{dt} \int_{\R^d} (u^k+v^k) dx \\
&+\frac{4k(k-1)m_1}{(k+m_1-1)^2} \int_{\R^d} \left| \nabla u^{\frac{k+m_1-1}{2}} \right|^2 dx+k\int_{\R^d} u^{k+1} dx+
\frac{k(k-1)}{k+\sigma_1-1}\int_{\R^d} u^{k+\sigma_1-1} c dx \\
&+ \frac{4k(k-1)m_2}{(k+m_2-1)^2} \int_{\R^d} \left| \nabla v^{\frac{k+m_2-1}{2}} \right|^2 dx+k\int_{\R^d} v^{k+1} dx+
\frac{k(k-1)}{k+\sigma_2-1}\int_{\R^d} v^{k+\sigma_2-1} z dx \\
= &N_1 k\int_{\R^d} u^k dx+\frac{k(k-1)}{k+\sigma_1-1} a_{11} \int_{\R^d} u^{k+\sigma_1} dx+ \frac{k(k-1)}{k+\sigma_1-1} a_{12} \int_{\R^d} u^{k+\sigma_1-1}v dx \\
&+N_2 k\int_{\R^d} v^k dx+\frac{k(k-1)}{k+\sigma_2-1} a_{22} \int_{\R^d} v^{k+\sigma_2} dx+ \frac{k(k-1)}{k+\sigma_2-1} a_{21} \int_{\R^d} v^{k+\sigma_2-1} u dx.
\end{align*}
By Young's inequality and the interpolation inequality, we further have
\begin{align}\label{260705}
& \frac{d}{dt} \int_{\R^d} (u^k+v^k) dx+\frac{4k(k-1)m_1}{(k+m_1-1)^2} \int_{\R^d} \left| \nabla u^{\frac{k+m_1-1}{2}} \right|^2 dx+k\int_{\R^d} u^{k+1} dx \nonumber \\
&+ \frac{4k(k-1)m_2}{(k+m_2-1)^2} \int_{\R^d} \left| \nabla v^{\frac{k+m_2-1}{2}} \right|^2 dx+k\int_{\R^d} v^{k+1} dx \nonumber \\
\le & \int_{\R^d} u^{k+\sigma_1} dx+C(\|u\|_1)+C \int_{\R^d} u^{k+\sigma_1} dx+C \int_{\R^d} v^{k+\sigma_1} dx \nonumber \\
&+\int_{\R^d} v^{k+\sigma_2} dx+C(\|v\|_1)+C \int_{\R^d} v^{k+\sigma_2} dx+C \int_{\R^d} u^{k+\sigma_2} dx.
\end{align}
Without loss of generality, we assume
\begin{align}
\sigma_2>\sigma_1,
\end{align}
from which we can infer from \eqref{260705} that
\begin{align}\label{2607051}
& \frac{d}{dt} \int_{\R^d} (u^k+v^k) dx+\frac{4k(k-1)m_1}{(k+m_1-1)^2} \int_{\R^d} \left| \nabla u^{\frac{k+m_1-1}{2}} \right|^2 dx+k\int_{\R^d} u^{k+1} dx \nonumber \\
& +\frac{4k(k-1)m_2}{(k+m_2-1)^2} \int_{\R^d} \left| \nabla v^{\frac{k+m_2-1}{2}} \right|^2 dx+k\int_{\R^d} v^{k+1} dx \nonumber \\
\le & C+C \int_{\R^d} u^{k+\sigma_2} dx+C \int_{\R^d} v^{k+\sigma_2} dx.
\end{align}
Letting
\begin{align}\label{33}
w=u^{\frac{k+m_1-1}{2}}, q=\frac{2(k+\sigma_2)}{k+m_1-1}, r=\frac{2k'}{k+m_1-1}
\end{align}
in Lemma \ref{GNS1} with $k>\frac{2 \sigma_2-p(m_1-1)}{p-2}$ (which is $q<p$) and $\frac{p}{p-2}(\sigma_2-m_1+1)<k'<k+1<k+\sigma_2$ (which is $\frac{q}{r}<\frac{2}{r}+1-\frac{2}{p}$), we have
\begin{align}\label{07052}
C \int_{\R^d} u^{k+\sigma_2} dx \le \frac{2k(k-1)m_1}{(k+m_1-1)^2} \left\| \nabla u^{\frac{k+m_1-1}{2}} \right\|_2^2+C(k) \|u\|_{k'}^{b_\alpha},
\end{align}
where
\begin{align}
b_\alpha=\frac{(k+\sigma_2)(1-\lambda)}{1-\frac{\lambda q}{2}},\quad \lambda=\frac{\frac{k+m_1-1}{2k'}-\frac{k+m_1-1}{2(k+\sigma_2)}}{\frac{k+m_1-1}{2k'}-\frac{1}{p}}.
\end{align}
We further apply $1<k'<k+1$ to obtain
\begin{align}\label{22}
C \int_{\R^d} u^{k+\sigma_2} dx \le \frac{2k(k-1)m_1}{(k+m_1-1)^2} \left\| \nabla u^{\frac{k+m_1-1}{2}} \right\|_2^2+C(k) \|u\|_{k+1}^{b_\alpha \theta} \|u\|_1^{b_\alpha (1-\theta)},
\end{align}
where $\theta=\frac{1-\frac{1}{k'}}{1-\frac{1}{k+1}}$. To use Young's inequality, we need the following condition:
\begin{align}\label{11}
b_\alpha \theta<k+1,
\end{align}
which is equivalent to
\begin{align}
(k+\sigma_2)(1-\lambda)\left(1-\frac{1}{k'}\right)<k-\frac{\lambda qk}{2}.
\end{align}
Substituting $\lambda$ into the above formula, one has
\begin{align*}
&\left( \frac{k+m_1-1}{2k'}-\frac{1}{p}\right) k'\sigma_2-\left( \frac{k+m_1-1}{2k'}-\frac{1}{p}\right)(k+\sigma_2) \\
&-\left( \frac{k+m_1-1}{2k'}-\frac{k+m_1-1}{2(k+\sigma_2)}\right) k'(k+\sigma_2)+\left( \frac{k+m_1-1}{2k'}-\frac{k+m_1-1}{2(k+\sigma_2)}\right)(k+\sigma_2) \\
< & \frac{kk'}{2}-\frac{k(k+\sigma_2)}{2}.
\end{align*}
This can be rewritten as
\begin{align}
(m_1-1-\sigma_2+2\sigma_2/d)(k'-1)<(m_1-1-\sigma_2+2/d)k.
\end{align}
Hence if 
\begin{align}\label{m1}
m_1>1+\sigma_2-2/d,
\end{align}
then \eqref{11} holds true for
\begin{align}
k>\frac{m_1-1-\sigma_2+2\sigma_2/d}{m_1-1-\sigma_2+2/d}(k'-1)
\end{align}
for any $\max\left(\frac{p}{p-2}(\sigma_2-m_1+1),1\right)<k'<k+1.$ Now combing back to \eqref{22}, we infer from \eqref{11} by using Young's inequality that
\begin{align}\label{44}
C \int_{\R^d} u^{k+\sigma_2} dx \le & \frac{2k(k-1)m_1}{(k+m_1-1)^2} \left\| \nabla u^{\frac{k+m_1-1}{2}} \right\|_2^2+C(k) \|u\|_{k+1}^{b_\alpha \theta} \|u\|_1^{b_\alpha (1-\theta)} \nonumber \\
\le & \frac{2k(k-1)m_1}{(k+m_1-1)^2} \left\| \nabla u^{\frac{k+m_1-1}{2}} \right\|_2^2+\frac{k}{2} \|u\|_{k+1}^{k+1}+C(\|u\|_1). 
\end{align}
Following similar arguments as in \eqref{33}-\eqref{44}, we obtain the estimates for $v$: if
\begin{align}\label{m2}
m_2>1+\sigma_2-2/d,
\end{align}
then for
\begin{align}
k>\frac{m_2-1-\sigma_2+2\sigma_2/d}{m_2-1-\sigma_2+2/d}(\overline{k}-1)
\end{align}
with any $\max\left(\frac{p}{p-2}(\sigma_2-m_2+1),1\right)<\overline{k}<k+1,$ the following holds true:
\begin{align}\label{55}
C \int_{\R^d} v^{k+\sigma_2} dx \le \frac{2k(k-1)m_2}{(k+m_2-1)^2} \left\| \nabla v^{\frac{k+m_2-1}{2}} \right\|_2^2+\frac{k}{2} \|v\|_{k+1}^{k+1}+C(\|v\|_1). 
\end{align}
Therefore, plugging \eqref{44} and \eqref{55} into \eqref{2607051}, we thus end up with 
\begin{align}\label{66}
&\frac{d}{dt} \int_{\R^d} (u^k+v^k) dx+\frac{4k(k-1)m_1}{(k+m_1-1)^2} \int_{\R^d} \left| \nabla u^{\frac{k+m_1-1}{2}} \right|^2 dx+k\int_{\R^d} u^{k+1} dx \nonumber \\
&+ \frac{4k(k-1)m_2}{(k+m_2-1)^2} \int_{\R^d} \left| \nabla v^{\frac{k+m_2-1}{2}} \right|^2 dx+k\int_{\R^d} v^{k+1} dx \nonumber \\
\le & C(k,\|u\|_1,\|v\|_1)
\end{align}
for any $$k> \max\left(\frac{m_1-1-\sigma_2+2\sigma_2/d}{m_1-1-\sigma_2+2/d}(k'-1), \frac{m_2-1-\sigma_2+2\sigma_2/d}{m_2-1-\sigma_2+2/d}(\overline{k}-1)\right)$$
with any $k'>\max\left(\frac{p}{p-2}(\sigma_2-m_1+1),1\right)$ and $\overline{k}>\max\left(\frac{p}{p-2}(\sigma_2-m_2+1),1\right).$ 

Consequently, combining \eqref{m1} and \eqref{m2}, we obtain that for
\begin{align}\label{77}
\min(m_1,m_2)>1+\max(\sigma_1,\sigma_2)-2/d,
\end{align}
the following estimates hold for any $0<t<\infty,$
\begin{align}
& \int_{\R^d} u^k dx+\int_{\R^d} v^k dx <\infty, \label{07090}\\
& \int_0^t \int_{\R^d} \left| \nabla u^{\frac{k+m_1-1}{2}} \right|^2 dxds+\int_0^t \int_{\R^d} \left| \nabla v^{\frac{k+m_2-1}{2}} \right|^2 dxds<\infty
\end{align}
for any $1 <k<\infty$. 

Based on \eqref{07090}, we apply the Moser iteration method to establish the uniform boundedness of the weak solution, 
\begin{align}\label{88}
\|u\|_{L^\infty(\R^d)}+\|v\|_{L^\infty(\R^d)} \le C\left(\|u_0\|_{L^1 \cap L^\infty(\R^d)}, \|v_0\|_{L^1 \cap L^\infty(\R^d)}\right),\quad \text{for any } t>0.
\end{align}
By Lemma \ref{ueps}, there exists a time $T_w>0$ and a solution $(u,v)$ to \eqref{uvsystem} on $[0,T_w)$ with initial data \eqref{initialdata}. The uniform bound \eqref{88} ensures, again by Lemma \ref{ueps}, that the solution can be extended globally in time, i.e. $T_w=\infty.$ Hence, we obtain a global weak solution $(u,v)$ to \eqref{uvsystem} satisfying the following regularities:
\begin{align}
& \|u\|_{L^\infty(0,\infty;L^1\cap L^\infty(\R^d))}+\|v\|_{L^\infty(0,\infty;L^1\cap L^\infty(\R^d))}<\infty, \label{07101}\\
& \left\| \nabla u^{\frac{k+m_1-1}{2}} \right\|_{L^2(0,\infty;L^2(\R^d))}+\left\| \nabla v^{\frac{k+m_2-1}{2}} \right\|_{L^2(0,\infty;L^2(\R^d))}<\infty, \label{07102}\\
&\|u\|_{L^{k+1}(0,\infty;L^{k+1}(\R^d))}+\|v\|_{L^{k+1}(0,\infty; L^{k+1}(\R^d))}<\infty \label{07103}
\end{align}
for any $1<k<\infty.$  \quad $\square$
\end{proof}

\indent\textbf{Proof of Theorem \ref{th1}.} Combining Proposition \ref{prop1} with the blow-up criterion in Lemma \ref{ueps} completes the proof of Theorem \ref{th1}.

\section{Long time behavior: proof of Theorem \ref{th2}}\label{longtime}
The constant steady state of \eqref{uvsystem} is given by
\begin{align}
u_*=N_1,\quad v_*=N_2, \quad c_*=a_{11} N_1+a_{12} N_2,\quad z_*=a_{21} N_1+a_{22} N_2.
\end{align}
In this section, we investigate the convergence of the global solution obtained in Theorem \ref{th1} toward the constant equilibrium. The analysis relies on an energy functional motivated by \cite{Bai16}.

Let us define
\begin{align}
A(t)=\int_{\R^d} \left(u-N_1-N_1 \ln \frac{u}{N_1}\right) dx, \label{At}\\
B(t)=\int_{\R^d} \left(v-N_2-N_2 \ln \frac{v}{N_2}\right) dx. \label{Bt}
\end{align}
We first establish an energy inequality for $A(t)+B(t)$. 
\begin{lemma}\label{lem1}
Let $m_1>1,m_2>1,\sigma_1=\frac{m_1+1}{2}, \sigma_2=\frac{m_2+1}{2}$. Under assumptions \eqref{initialdata} and 
\begin{align}
&m_1>\frac{m_2-1}{2}+2-2/d \\
\text{or}\quad & m_2>\frac{m_1-1}{2}+2-2/d.
\end{align} 
Suppose also that the positive numbers $a_{ij}$ satisfy 
\begin{align}
\left\{
  \begin{array}{ll}
    (a_{11}^2+a_{12}^2)N_1<16 m_1, \\[1mm]
    \frac{(a_{11}^2+a_{12}^2)N_1}{m_1}+\frac{(a_{21}^2+a_{22}^2)N_2}{m_2}<16+\frac{(a_{11}a_{22}-a_{12}a_{21})^2 N_1N_2}{16m_1m_2}.
  \end{array}
\right.
\end{align}
Then there exists $\delta>0$ such that 
\begin{align}\label{0630}
\frac{d}{dt} (A(t)+B(t)) \le -\delta F(t), 
\end{align}
where
\begin{align}\label{Ft}
F(t)= & \int_{\R^d} \left( \nabla u^{\frac{m_1-1}{2}}-\frac{m_1-1}{4m_1} \nabla c \right)^2+ \left( \nabla v^{\frac{m_2-1}{2}}-\frac{m_2-1}{4m_2} \nabla z \right)^2 dx \nonumber \\
& + \int_{\R^d} (u-N_1)^2+(v-N_2)^2+(c-c_*)^2+(z-z_*)^2 dx.
\end{align}
\end{lemma}
\begin{proof}
We introduce the constants
\begin{align}\label{cons}
a=\frac{4m_1 N_1}{(m_1-1)^2},~~b=\frac{2N_1}{m_1-1},~~\alpha=\frac{4m_2 N_2}{(m_2-1)^2},~~\beta=\frac{2N_2}{m_2-1}.
\end{align}
A direct computation gives
\begin{align*}
A'(t)=&-a \int_{\R^d} \left| \nabla u^{\frac{m_1-1}{2}} \right|^2 dx+b \int_{\R^d} \nabla u^{\frac{m_1-1}{2}} \cdot \nabla c dx-\int_{\R^d} (u-N_1)^2 dx \\
=&-a \int_{\R^d} \left( \nabla u^{\frac{m_1-1}{2}}-\frac{b}{2a} \nabla c \right)^2 dx+\frac{b^2}{4a} \int_{\R^d} |\nabla c|^2 dx-\int_{\R^d} (u-N_1)^2 dx \\
=&-a \int_{\R^d} \left( \nabla u^{\frac{m_1-1}{2}}-\frac{b}{2a} \nabla c \right)^2 dx -\bigg(-\frac{a_{11}b^2}{4a} \int_{\R^d} (u-N_1)(c-c_*) dx \\
&-\frac{a_{12}b^2}{4a} \int_{\R^d} (v-N_2)(c-c_*) dx+ \frac{b^2}{4a}\int_{\R^d}(c-c_*)^2 dx +\int_{\R^d}(u-N_1)^2 dx \bigg).
\end{align*}
Similarly, we have
\begin{align*}
B'(t)=& -\alpha \int_{\R^d} \left| \nabla v^{\frac{m_2-1}{2}} \right|^2 dx+\beta \int_{\R^d} \nabla v^{\frac{m_2-1}{2}}\cdot \nabla z dx-\int_{\R^d} (v-N_2)^2 dx \\
= & -\alpha \int_{\R^d} \left( \nabla v^{\frac{m_2-1}{2}}-\frac{\beta}{2\alpha} \nabla z \right)^2 dx +\frac{\beta^2}{4\alpha} \int_{\R^d} |\nabla z|^2 dx-\int_{\R^d} (v-N_2)^2 dx \\
=&-\alpha \int_{\R^d} \left( \nabla v^{\frac{m_2-1}{2}}-\frac{\beta}{2\alpha} \nabla z \right)^2 dx -\bigg(-\frac{a_{21} \beta^2}{4\alpha} \int_{\R^d} (u-N_1)(z-z_*) dx \\
&-\frac{a_{22}\beta^2}{4\alpha} \int_{\R^d} (v-N_2)(z-z_*) dx+\frac{\beta^2}{4\alpha}\int_{\R^d}(z-z_*)^2 dx +\int_{\R^d}(v-N_2)^2 dx \bigg).
\end{align*}
Taking $A'(t)$ and $B'(t)$ into account, we obtain
\begin{align}\label{07095}
&\frac{d}{dt} (A(t)+B(t)) \nonumber \\
=&-a \int_{\R^d} \left( \nabla u^{\frac{m_1-1}{2}}-\frac{b}{2a} \nabla c \right)^2 dx-\alpha \int_{\R^d} \left( \nabla v^{\frac{m_2-1}{2}}-\frac{\beta}{2\alpha} \nabla z \right)^2 dx-\int_{\R^d} XPX^T dx,
\end{align}
where the vector $X$ and the matrix $P$ are defined as
\begin{align*}
X=\left( u-N_1,v-N_2,c-c_*,z-z_* \right), \\
P=\begin{pmatrix}
                                                   1 & 0 & - \frac{a_{11}b^2}{8a} & -\frac{a_{21} \beta^2}{8\alpha} \\
                                                   0 & 1 & -\frac{a_{12} b^2}{8a} & -\frac{a_{22} \beta^2}{8 \alpha} \\
                                                   -\frac{a_{11} b^2}{8a} & -\frac{a_{12} b^2}{8a} & \frac{b^2}{4a} & 0 \\
                                                   -\frac{a_{21} \beta^2}{8\alpha} & -\frac{a_{22} \beta^2}{8 \alpha} & 0 & \frac{\beta^2}{4\alpha} \\
                                                 \end{pmatrix}.
\end{align*}
The key step in establishing \eqref{0630} lies in proving that $P$ is positive definite. Using Sylvester's criterion, we claim that $P$ is positive definite provided that
\begin{align}
\left\{
  \begin{array}{cc}
  (a_{11}^2+a_{12}^2)b^2<16a, \\[1mm]
16+ \frac{(a_{11}a_{22}-a_{12}a_{21})^2 b^2 \beta^2}{16a \alpha}-\frac{a_{11}^2 b^2}{a}-\frac{a_{12}^2 b^2}{a}-\frac{a_{21}^2 \beta^2}{\alpha}-\frac{a_{22}^2 \beta^2}{\alpha}>0.
  \end{array}
\right.
\end{align}
Recalling \eqref{cons}, the above conditions are equivalent to
\begin{align}
\left\{
  \begin{array}{ll}
    (a_{11}^2+a_{12}^2)N_1<16 m_1, \\
    \frac{(a_{11}^2+a_{12}^2)N_1}{m_1}+\frac{(a_{21}^2+a_{22}^2)N_2}{m_2}<16+\frac{(a_{11}a_{22}-a_{12}a_{21})^2 N_1N_2}{16m_1m_2}.
  \end{array}
\right.
\end{align}
Consequently, there exists $\delta_0>0$ such that
\begin{align}
XPX^T \ge \delta_0 |X|^2.
\end{align}
Returning to \eqref{07095}, we conclude that
\begin{align*}
&\frac{d}{dt} (A(t)+B(t)) \\
 \le &-a \int_{\R^d} \left( \nabla u^{\frac{m_1-1}{2}}-\frac{b}{2a} \nabla c \right)^2 dx-\alpha \int_{\R^d} \left( \nabla v^{\frac{m_2-1}{2}}-\frac{\beta}{2\alpha} \nabla z \right)^2 dx \\
&-\delta_0 \int_{\R^d} (u-N_1)^2+(v-N_2)^2+(c-c_*)^2+(z-z_*)^2 dx,
\end{align*}
Choosing $\delta=\min\{a,\alpha,\delta_0\}$, we arrive at \eqref{0630} and thus complete the proof.  \quad $\square$
\end{proof}

As a direct application of the energy inequality \eqref{0630}, we obtain the following convergence result.
\begin{lemma}\label{lem2}
Under the assumptions of Lemma \ref{lem1}, let $(u,v)$ be the global solution obtained in Theorem \ref{th1}. Then there exist constants $C>0$ and $\lambda>0$ such that
\begin{align}
\int_{\R^d} (u-N_1)^2+(v-N_2)^2 dx \le C e^{-\lambda t} \quad \text{for all } t>t_0
\end{align}
for some $t_0>0$.
\end{lemma}
\begin{proof}
Define 
\begin{align}
f(t):=\int_{\R^d}  (u-N_1)^2+(v-N_2)^2 dx.
\end{align}
The proof can be divided into two steps. First, we prove that $f(t)$ tends to zero as time goes to infinity. Then, we show that the global solution $(u,v)$ converges exponentially in time to the constant equilibrium $(N_1,N_2)$.

{\it\textbf{Step 1}} (Asymptotic behavior of $f(t)$) \quad We first define an auxiliary function 
\begin{align}\label{gx}
g(x)=x-x_*-x_* \ln \frac{x}{x_*}.
\end{align} 
A simple computation yields that $g(x) \ge 0$ for all $x>0$. Therefore, recalling the definitions of $A(t)$ and $B(t)$ given in \eqref{At} and \eqref{Bt}, we have
\begin{align}
A(t) \ge 0, \quad B(t) \ge 0.
\end{align}
On the other hand, $F(t)$ as given by \eqref{Ft} satisfies
\begin{align}
F(t) \ge f(t)\quad \text{for all } t>0.
\end{align}
Hence, we infer from Lemma \ref{lem1} that
\begin{align}\label{071010}
\frac{d}{dt} (A(t)+B(t)) \le -\delta f(t) \quad \text{for all } t>0.
\end{align}
It follows that
\begin{align}\label{07108}
\int_1^\infty f(t) dt \le \frac{1}{\delta} (A(1)+B(1)-A(\infty)-B(\infty)) \le \frac{1}{\delta} (A(1)+B(1))<\infty.
\end{align}
Next, we compute
\begin{align}
f_t =&2 \int_{\R^d} (u-N_1)u_t+(v-N_2)v_t dx \nonumber \\
=& 2 \int_{\R^d} \left( -\frac{4m_1}{(m_1+1)^2} \left| \nabla u^{\frac{m_1+1}{2}} \right|^2+\frac{1}{\sigma_1+1} u^{\sigma_1+1} (-\Delta c)-u(u-N_1)^2 \right)dx \nonumber \\
& + 2 \int_{\R^d} \left( -\frac{4m_2}{(m_2+1)^2} \left| \nabla v^{\frac{m_2+1}{2}} \right|^2+\frac{1}{\sigma_2+1} v^{\sigma_2+1} (-\Delta z)-v(v-N_2)^2 \right)dx.
\end{align}
Thanks to the regularities \eqref{07101}-\eqref{07103}, it holds that
\begin{align}\label{07140}
\int_0^\infty |f_t|dt<\infty.
\end{align}
Combining \eqref{07140} with \eqref{07108}, we arrive at
\begin{align}\label{ft}
f(t) \to 0 \quad \text{as}\quad  t \to \infty.  
\end{align}

{\it\textbf{Step 2}} (Exponential decay in time of $f(t)$) \quad Recalling the function \eqref{gx}, we have the asymptotic estimate
\begin{align}
g(x) \sim \frac{1}{2 x_*} (x-x_*)^2 \quad \text{as} \quad x \to x_*.
\end{align}
Since $\int_{\R^d}(u-N_1)^2dx \to 0$ as $t \to \infty$, we can choose $t_0>0$ such that
\begin{align}\label{071011}
\frac{1}{3N_1} \int_{\R^d} (u-N_1)^2 dx \le \int_{\R^d} \left(u-N_1-N_1 \ln \frac{u}{N_1}\right) dx \le \frac{1}{N_1} \int_{\R^d} (u-N_1)^2 dx \text{ for all } t>t_0.
\end{align}
Similarly, we have
\begin{align}\label{071012}
\frac{1}{3N_2} \int_{\R^d} (v-N_2)^2 dx \le \int_{\R^d} \left(v-N_2-N_2 \ln \frac{v}{N_2}\right) dx \le \frac{1}{N_2} \int_{\R^d} (v-N_2)^2 dx \text{ for all } t>t_0.
\end{align}
Hence, we infer from \eqref{071010} that
\begin{align}
\frac{d}{dt} (A(t)+B(t)) \le -\delta f(t) \le -\delta_1 (A(t)+B(t)) \text{ for all } t>t_0,
\end{align}
which implies
\begin{align}
A(t)+B(t) \le c_1 e^{-c_2 t} \text{ for all } t>t_0.
\end{align}
Therefore, owing to \eqref{071011} and \eqref{071012}, we conclude that
\begin{align}
f(t) \le c_3 \left(A(t)+B(t) \right) \le C e^{-c_2 t} \text{ for all } t>t_0.
\end{align}
This completes the proof. \quad $\square$
\end{proof}

\indent\textbf{Proof of Theorem \ref{th2}.} 
Combining Lemma \ref{lem1} with Lemma \ref{lem2} establishes Theorem \ref{th2}.

\end{document}